\documentclass[11pt]{article}
\usepackage{amsmath}
\usepackage{graphicx,psfrag,epsf}
\usepackage{enumerate}
\usepackage{natbib}
\usepackage{amsmath,amsthm,amsfonts,epsfig,amssymb,natbib,eucal,eufrak}

\usepackage{booktabs}
\usepackage{multirow}
\usepackage{siunitx}
\usepackage{qtree}
\usepackage{hyperref}

\usepackage{float}
\restylefloat{table}
\usepackage{color}

\usepackage{pdfpages}
\usepackage{wrapfig}

\usepackage{comment}
\usepackage[dvipsnames]{xcolor}

\usepackage{tikz}

\newtheorem{lem}{Lemma}

\newtheorem{ex}{Example}

\newtheorem{remark}{Remark}

\newcommand{\blind}{0}

\def\baselinestretch{1.19}

\begin{document}

\def\spacingset#1{\renewcommand{\baselinestretch}%
{#1}\small\normalsize} \spacingset{1}


\if0\blind
{\title{\bf {The Water Puzzle and Marginal Utility Optimization}}
\author{Yaakov Malinovsky
\thanks{Department of Mathematics and Statistics, University of Maryland,
Baltimore County, Baltimore, MD 21250, USA.
Email: {\tt yaakovm@umbc.edu}.  Research supported in part
by BSF grant 2020063.}
\,\,\,
and\,\,\,
Isaac M. Sonin\thanks{
Department of Mathematics and Statistics, University of North Carolina at Charlotte, Charlotte,
NC 28223, USA.
Email: {\tt imsonin@uncc.edu}.}
}
 \maketitle
} \fi
\if1\blind
{
  \bigskip
  \bigskip
  \bigskip
  \begin{center}
    {\LARGE\bf The Water Puzzle and Marginal Utility Optimization }
\end{center}
  \medskip
} \fi

\begin{abstract}
We present a variation of the water puzzle, which is related to a simple model of marginal utility. The problem has an intriguing solution and can be extended in several directions.
\end{abstract}
\noindent%
{\it Keywords: convexity, marginal utility, optimization}

\noindent%
{\it MSC2020: 00A08, 49 }

\spacingset{1.15} 
\section*{Water Puzzle Defying Intuition}
A few years ago, one of the authors of this note proposed the following problem for a math olympiad.
There are two cups of tea on a table, each with a two-unit capacity.
Cup 1 contains one unit of tea with an $80\%$ concentration, while cup 2 contains one unit of tea with a $25\%$ concentration.
You have one unit of hot water in your cup. Distribute it between these two cups: $x$ ($0\leq x\leq1$) into cup 1 and the rest, $1-x$, into cup 2.
You will then be given back the volume $x$ from cup 1 and the volume $1-x$ from cup 2. This means that you will get back one unit in total.
You're ready to drink your tea. You just need to figure out what value of $x$ will make your tea as strong as possible.

All of the participants, and a majority of the organizing mathematicians, in the True/False format got the wrong
answer $x=1$, probably with what seems a natural reasoning.
If $x=1$, we receive tea with $40\%$ concentration. Why mess with the lower-concentration cup?
As we will see below, the correct answer is with $x<1$ and maximal tea concentration larger than $40\%$.
Further, we show that if the smaller concentration is more than one-fourth of the larger concentration, then the cup with the lower concentration should be always used.

Let us consider a scenario in which a similar problem is encountered with three cups of tea, each with a respective concentration of $80\%$, $40\%$, and $25\%$.
Once more, one unit of hot water must be distributed between the three cups. The corresponding amounts are then returned from each cup, resulting in a total of one unit. It is interesting to note that, in this case, the cup with the lowest concentration should not be used. However, if one considers any two cups from the three cups, both will be used.

\section*{Marginal Utility Model}
Now, we formulate the maximization problem with a transparent socio-economic interpretation, which is equivalent to the $n$ cups problem. There are $n$ potential projects ranked by positive indices $a_1>a_2>\cdots>a_n>0$ and available unit resources that must be distributed among projects according to the vector $x=\left(x_1, x_2,\ldots,x_n\right)$. What is the optimal allocation if a decision maker wants to maximize her/his potential return? That is, the decision maker (investor) wants to maximize the function
\begin{align}
 \label{eq:O}
&	
F(x_1,\dots,x_n)=\sum_{i=1}^n \frac{a_ix_i}{1+x_i}:=\sum_{i=1}^n g_i(x_i),
\end{align}
subject to
\begin{align}
\label{eq:st}
&
x=(x_1,\ldots,x_n)\in S_n,\,\,\,
S_n=\left\{(x_1,\ldots,x_n):x_1+\cdots+x_n=1,\,\, x_i\geq 0,\,\, i=1,\ldots,n.\right\}.
\end{align}

Each function in the sum in \eqref{eq:O} has the form $\frac{ax}{1+x}$.
This function is smooth and concave downward.
Similar functions are used in economics textbooks to illustrate the concept of marginal utility and the law of diminishing marginal utility.
In our model, if vector $x$ is an admissible solution, then the marginal utility of investment $x_i$ into project $i$, $g_i^{\prime}(x_i)$, where $g_i^{\prime}(x_i)$ is a notation for the derivative of a function $g_i(x_i)$, is a decreasing function of $x_i$ with respect to investment growth.
Philosophers, political scientists, and economists developed both the concept and the law to explain consumers' motives, human behavior in general, and the economic reality of price.
This concept and the law have an intriguing history, which can be found in the classical book of \cite{K2016} and traces back to Aristotle ($4^{\text{th}}$ century BC) and Daniel Bernoulli ($18^{\text{th}}$ century). During the $18^{\text{th}}$ and $19^{\text{th}}$ centuries, it was discussed in different forms by many prominent scientists in economics, finance, and philosophy.
Later, it was revived by various $20^{\text{th}}$-century thinkers, with early contributions by Frank P. Ramsey,  John von Neumann, and Oskar Morgenstern, among many others. Nowadays, it is lay knowledge.
The problem of maximization in \eqref{eq:O}, subject to the condition in \eqref{eq:st}, is probably the simplest model where the decision maker has to solve the problem of optimal allocation of resources between $n$ projects with variable rates of return, depending on vector $x$.
Although our model seems quite natural, we were unsuccessful in finding any references to this or a similar optimization model in the literature.

\section*{Three Possible Solutions}
There are three different, although essentially equivalent, solutions. All three prove the following intuitively clear statements:
\par
A) the optimal solution exists and is unique;
B) for this optimal solution there is an ``active'' zone of investment that includes the first $k^{\star}$ projects,
$1\leq k^{\star}\leq n$, and $x_i=0$ for $k^{\star}<i\leq n$.
\par
Subsequently, one possible approach is that the marginal utility for all active projects must be equal, i.e. $g'_i(x_i)=constant>0$ for all $1\leq i \leq k^{\star}$. Citing the classical textbook in economics \cite[Chapter 5]{SN2020}: ``The fundamental condition of maximum satisfaction or utility is the Equimarginal principle. It states that a consumer will achieve maximum satisfaction or utility when the marginal utility of the last dollar spent on a good is exactly the same as the marginal utility of the last dollar spent on any other good''.  The corresponding algorithm to find $k^{\star}$ and the optimal vector $x$ is based on the sequential consideration of equalities $g^{\prime}_1(x_1)=g^{\prime}_2(x_2)$, $g^{\prime}_2(x_2)=g^{\prime}_3(x_3)$, etc., until $k^{\star}$ is reached.\par

The second possible approach, based on the concept of Lagrange multipliers  and convex optimization, was proposed by Andrey Dmitruk, who read the original draft of our paper. This approach requires a higher level of knowledge of optimization theory.

The third possible approach, based on convexity and elementary calculus, is presented in the next section.
This approach is elementary and easily accessible to undergraduates. We also provide all the necessary details to make it self-contained.
Furthermore, we obtain a closed-form solution to the problem.



\section*{Algorithm of Solution}
It is clear that the problem  of maximization \eqref{eq:O} is equivalent to the problem of minimization
\begin{align}
\label{eq:Con}
\min_{x\in S_n}\sum_{i=1}^n \frac{a_i}{1+x_i}
\end{align}
subject to
\begin{align*}
&
x=(x_1,\ldots,x_n)\in S_n,\,\,\,
S_n=\left\{(x_1,\ldots,x_n):x_1+\cdots+x_n=1,\,\, x_i\geq 0,\,\, i=1,\ldots,n.\right\}.
\end{align*}

We can express $x_1$ as $x_1=1-x_2-x_3-\cdots-x_{n}$ and reformulate
the optimization problem \eqref{eq:Con} as follows.
For the numbers
\begin{equation}
\label{eq:strict}
a_1>a_2>\cdots>a_n>0,
\end{equation}
we want to minimize
\begin{align}
 \label{eq:O1}
&	
f\left(x_2,\ldots,x_{n}\right):=\sum_{i=2}^{n}\frac{a_i}{1+x_i}+\frac{a_1}{2-x_2-\cdots-x_{n}}
\end{align}
subject to
\begin{align}
\label{eq:5}
&
x=(x_2,\ldots,x_{n})\in S^{\star}_{n-1},\,\,\,
S^{\star}_{n-1}=\left\{(x_2,\ldots,x_{n}):x_2+\cdots+x_{n}\leq 1,\,\, x_i\geq 0,\,\, i=2,\ldots,n\right\}.
\end{align}


We present below two lemmas which allow us to obtain an efficient and simple algorithm for solving \eqref{eq:O1} with constraints \eqref{eq:5}.
\begin{lem}
\label{l1}
The function $f(x)$ attains a unique global minimum on $S^{\star}_{n-1}$.
\end{lem}
\begin{proof}
The set $S^{\star}_{n-1}$ is a closed and bounded set in $\mathbb{R}^{n}$, and $f(x)$ is a convex continuous function in $\mathbb{R}^{n}$. Therefore, by the Weierstrass theorem, $f(x)$ attains its global minimum.
The uniqueness of this minimum follows from the strict convexity of the function $f(x)$.

\end{proof}

Define
\begin{align}
\label{eq:ref}
x^{\star}=\left(x^{\star}_2, x^{\star}_3,\ldots,x^{\star}_{n}\right)=\arg\min_{x\in S^{\star}_{n-1}}f\left(x_2,\ldots,x_{n}\right),\,\,\, x^{\star}_1=1-x^{\star}_2-\cdots-x^{\star}_{n}.
\end{align}
Due to Lemma~\ref{l1}, the solution \( x^{\star} \) to the problem \eqref{eq:ref} exists and is unique.
Next, we present a simple method to find it by exploiting the strict convexity of the function \( f(x) \).

We will need to consider the function \( f(x_2, \ldots, x_n) \), defined in \eqref{eq:O1}, in the special case where the last \( n - l \) variables are set to zero. Accordingly, we define
\begin{equation*}
f_{2:l}(x_2, \ldots, x_l) = f(x_2, \ldots, x_l, 0, \ldots, 0), \quad l = 2, \ldots, n,
\end{equation*}

and
\begin{align*}
{\displaystyle {\nabla f_{2:l}}(x_2, \ldots, x_l)=\left(\frac{\partial f_{2:l}(x_2, \ldots, x_l) }{\partial x_2},\ldots,\frac{\partial f_{2:l}(x_2, \ldots, x_l)}{\partial x_{l}}\right)},\,\, l=2,\ldots,n.
\end{align*}

\begin{lem}
\label{l2}
If the root of
\begin{equation}
\label{eq:grad_zero}
\nabla f_{2:n}(x_2, \ldots, x_n) = 0_{1 \times (n-1)}
\end{equation}
belongs to \( S^{\star}_{n-1} \), then this root is the optimal value \( x^{\star} \), which is an interior point of \( S^{\star}_{n-1} \) satisfying
\begin{equation}
\label{eq:strict_ineq}
1 > x^{\star}_2 > \cdots > x^{\star}_n > 0.
\end{equation}
Otherwise, \( x^{\star} \) is a boundary point of \( S^{\star}_{n-1} \). In this case, one of the following holds:
\begin{itemize}
\item[(i)]
There exists \( k \in \{2, \ldots, n-1\} \) such that
\begin{equation*}
\label{eq:boundary_k}
1 > x^{\star}_2 > \cdots > x^{\star}_{k} > x^{\star}_{k+1} = \cdots = x^{\star}_{n} = 0,
\end{equation*}
and
\begin{equation*}
\label{eq:grad_k}
\nabla f_{2:k}(x^{\star}_2, \ldots, x^{\star}_k) = 0_{1 \times (k-1)};
\end{equation*}
\item[(ii)]
If no such \( k \) exists, then
\begin{equation*}
\label{eq:all_zero}
x^{\star}_2 = \cdots = x^{\star}_n = 0.
\end{equation*}
\end{itemize}
\end{lem}

\begin{proof}
The following observation is key to determining \( x^{\star} \).
If at least one of \( x^{\star}_i \) and \( x^{\star}_{i+1} \) is nonzero, then
\begin{equation}
\label{eq:Ine}
x^{\star}_i > x^{\star}_{i+1}.
\end{equation}

To prove the inequality \eqref{eq:Ine}, suppose, to the contrary, that \( x^{\star}_i < x^{\star}_{i+1} \). Since \( a_i > a_{i+1} \), swapping \( x^{\star}_i \) and \( x^{\star}_{i+1} \) decreases the value of the objective function in \eqref{eq:O1} (by the rearrangement inequality), without violating the constraints in \eqref{eq:5}. This contradicts the optimality of \( x^{\star} \), and thus the inequality \eqref{eq:Ine} follows.

From the strict convexity of \( f(x) \) on \( S^{\star}_{n-1} \), it follows that if the root of \eqref{eq:grad_zero} lies in \( S^{\star}_{n-1} \), then it is the optimal value \( x^{\star} \).
Since \( a_1 > a_2 > \cdots > a_n > 0 \), this value is an interior point of \( S^{\star}_{n-1} \), satisfying \eqref{eq:strict_ineq} by \eqref{eq:Ine}.
Otherwise, \( x^{\star} \) lies on the boundary of \( S^{\star}_{n-1} \); that is, at least one component \( x^{\star}_i \) is equal to zero.
Appealing to \eqref{eq:Ine}, it follows that \( x^{\star}_n = 0 \).

Now, setting \( k = n-1 \), we again check—by strict convexity—whether the root of \eqref{eq:grad_k}, together with the condition \( x^{\star}_n = 0 \), lies in \( S^{\star}_{n-1} \).
If so, we obtain case (i) with \( k = n-1 \); otherwise, by \eqref{eq:Ine}, we conclude that \( x^{\star}_{n-1} = 0 \) also, and we continue in the same manner.
\end{proof}

First, we demonstrate the solution and illustrate how Lemmas~\ref{l1} and~\ref{l2} are applied to solve problem~\eqref{eq:O} subject to the constraints in~\eqref{eq:st}, in the simplest case, \( n = 2 \).

\begin{ex}[Two cups]
In the case \( n = 2 \), equations \eqref{eq:O1} and \eqref{eq:5} reduce to
\begin{align*}
f(x_2) = \frac{a_2}{1 + x_2} + \frac{a_1}{2 - x_2}, \quad 0 \leq x_2 \leq 1.
\end{align*}
By Lemma~\ref{l1}, the function \( f(x_2) \) attains a unique global minimum on the interval \([0, 1]\). By \eqref{eq:grad_zero}, we have \( \nabla f_{2:n}(x_2, \ldots, x_n) = 0 \)
if and only if
\[
x_2 = \frac{2\sqrt{a_2} - \sqrt{a_1}}{\sqrt{a_1} + \sqrt{a_2}}.
\]
Since \( a_1 > a_2 > 0 \), it follows that \( x_2 < 1 \). Moreover, \( x_2 > 0 \) if and only if \( a_2 > a_1 / 4 \). Therefore, by Lemma~\ref{l2}, the optimal solution is given by
\begin{align*}
\label{eq:sol}
(x^{\star}_{1}, x^{\star}_2)
=\left\{
\begin{array}{lll}
    \left(\frac{2\sqrt{a_1} - \sqrt{a_2}}{\sqrt{a_1} + \sqrt{a_2}},\, \frac{2\sqrt{a_2} - \sqrt{a_1}}{\sqrt{a_1} + \sqrt{a_2}}\right) & \text{if}  & a_2 > a_1 / 4, \\
    \left(1,\, 0\right) & \text{otherwise}.
\end{array}
\right.
\end{align*}

Recalling the example with two cups of tea at respective concentrations of \( 80\% \) and \( 25\% \), we obtain \( x^{\star}_1 = 0.9242\ldots \), \( x^{\star}_2 = 0.0757\ldots \), and a maximal tea concentration (see \eqref{eq:O}) of \( F(x^{\star}_1, x^{\star}_2) = 0.4018\ldots \), i.e., approximately \( 40.2\% \).
\end{ex}

An obvious observation is that if at least one cup is not used, and we add a new cup whose concentration is no greater than the concentrations of the unused cups, then this additional cup should also not be used.
Combining this fact with Lemma~\ref{l2} and its proof, and recalling that $x^{\star}_1=1-x^{\star}_2-\cdots-x^{\star}_{n}$ , we obtain the optimal solution to problem~\eqref{eq:Con}, which we present below as a sequential algorithm.

\bigskip

Define $$e_i=\sqrt{\frac{a_i}{a_{1}}},\,\,\,\,i=2,\ldots,n,\,\,\,\,e_1:=1.$$

\begin{center}
{\bf ALGORITHM}
\end{center}
\begin{itemize}
\item
Find sequentially (and stop at the first violation)
$$k^{\star}_n=\max \left\{2\leq j \leq n: e_j>\frac{e_{j-1}+\cdots+e_1}{j}\right\}.$$
\item
If the set in the parentheses is empty, then the optimal solution is $x^{\star}_{1}=1, x^{\star}_{2}=\cdots=x^{\star}_{n}=0$.
If the set in the parentheses is not empty, then the optimal solution is
\begin{equation*}
x^{\star}_{j}=\left\{\begin{array}{ccc}
                       \frac{k^{\star}_n\times{e_j}-\sum_{i=1, i\ne j}^{k^{\star}_n}{e_i}}{{e_1}+{e_2}+\cdots+{e_{k^{\star}_n}}} & for & j=1,\ldots, k^{\star}_n \\
                       0 & for & j=k^{\star}_n+1,\ldots,n.
                     \end{array}
                     \right.
\end{equation*}
\end{itemize}

\begin{remark}
If some of the inequalities in~\eqref{eq:strict} are not strict, then certain inequalities in \eqref{eq:Ine} and Lemma~\ref{l2} may also become non-strict.
As a result, some nonzero components of the optimal solution may be equal.
Nevertheless, the same algorithm can still be applied without modification.
\end{remark}


\begin{ex}[ three cups]
The algorithm implementation supplies the solution:
\begin{align*}
&
\left(x^{\star}_{1}, x^{\star}_{2}, x^{\star}_{3}\right)=
\left\{
\begin{array}{lll}
  \left(\frac{3e_1-e_2-e_3}{e_1+e_2+e_3},\,\frac{3e_2-e_1-e_3}{e_1+e_2+e_3}, \,\frac{3e_3-e_1-e_2}{e_1+e_2+e_3}\right) & {\text{if}} & 3e_3>e_1+e_2 \\
  \left(\frac{2e_1-e_2}{e_1+e_2},\,\frac{2e_2-e_1}{e_1+e_2},\,0\right)  &\text{if} & \,2e_2>e_1.\\
  \left(1,\,0,\,0\right) & & otherwise.
\end{array}
\right.
\end{align*}
\end{ex}
Recalling the example with three cups of tea at respective concentrations of  $80\%, 40\%$, and $25\%$, we obtain $x^{\star}_{1}=0.7573\ldots,\,\, x^{\star}_{2}=0.2426\ldots,\,\, x^{\star}_{3}=0$ and a maximal tea concentration (see \eqref{eq:O}) of $F(x^{\star}_1, x^{\star}_2, x^{\star}_3)=0.4228...$, i.e. about $42.3\%$.

\section*{Historical Remarks and Open Problems}
It is worth noting that water puzzles have a long history. According to \cite{KN2021}, some of these puzzles are believed to date back to medieval times.
The model described above can be extended in several directions. One is through adding more constraints on vector $x$. For example, say there are some ``priority''
projects, so for some coordinates we choose bounds $b_i$ and $d_i$ and require that $0<b_i\leq x_i\leq d_i<1$. Another interesting extension is to consider a stochastic modification and to identify finance-economic problems of optimization that can be reduced to the problem analyzed in this paper, at least locally.

An especially compelling extension of our model involves a game-theoretical formulation in which $N$ participants, each possessing potentially different amounts of resources, allocate these resources across $n$ available projects governed by a modified function $F$. This scenario is analogous to a setting with $n$ cups and $N$ players, where each player, equipped with a cup of pure water of a specific volume, seeks to maximize the strength of her tea.

\section*{Acknowledgements}
We thank Yosef Rinott and Andrey Dmitruk for their insightful remarks and valuable suggestions on the original draft. We also gratefully acknowledge the three referees for their helpful comments, which led to significant improvements in the presentation of this paper.

\end{document}